\documentclass[11pt]{article}
\usepackage{amsthm,amssymb,amsmath,lineno}
\usepackage{xcolor}
\usepackage{epsfig}
\usepackage{color}
\usepackage{pgf,tikz}
\usetikzlibrary{arrows}
\allowdisplaybreaks

\newtheorem{prelem}{{\bf Theorem}}

\newtheorem{theorem}{Theorem}

\newtheorem{lemma}[theorem]{Lemma}

\theoremstyle{definition}

\theoremstyle{remark}

\title{Improved Lower Bounds on the General Reduced Second Zagreb Index of Trees and Unicyclic Graphs}

\author{Nasrin Dehgardi$^a$ \qquad Sandi Klav\v{z}ar$^{b, c, d}$ \qquad Mahdieh Azari$^{e}$\thanks{Corresponding author.}\vspace{3mm}\\
$^{a}$\small Department of Mathematics and Computer Science, \\
\small Sirjan University of Technology, Sirjan, Iran\\
\small {\tt n.dehgardi@sirjantech.ac.ir}\\
\small {\tt https://orcid.org/0000-0001-8214-6000}\vspace{2mm}\\
$^{b}$\small Faculty of Mathematics and Physics, University of Ljubljana, Slovenia\\
$^{c}$\small Institute of Mathematics, Physics and Mechanics, Ljubljana, Slovenia \\
$^{d}$\small Faculty of Natural Sciences and Mathematics, University of Maribor, Slovenia\\
\small {\tt sandi.klavzar@fmf.uni-lj.si}\\ \small {\tt https://orcid.org/0000-0002-1556-4744}\vspace{2mm}\\
$^{e}$\small Department of Mathematics,\\
\small Kaz.C., Islamic Azad University, Kazerun, Iran\\
\small {\tt  mahdieh.azari@iau.ac.ir, mahdie.azari@gmail.com} \\ \small {\tt https://orcid.org/0000-0002-0919-0598}}
\date{\today}

\begin{document}
\maketitle

\begin{abstract}
For a simple graph $\Gamma$ and a real number $\lambda$, the general reduced second Zagreb index is defined by the formula
$$GRM_\lambda(\Gamma)=\sum_{ab\in E(\Gamma)}[(\deg_{\Gamma}(a)+\lambda)(\deg_{\Gamma}(b)+\lambda)]\,.$$
A sharp lower bound for $GRM_\lambda$ over all trees of given order and maximum degree under the condition that $\lambda\ge -\frac{1}{2}$ is established. A parallel result is proved for unicyclic graphs under the condition $\lambda\ge -\frac{1}{2}$. The corresponding minimal trees and unicyclic graphs are identified. These findings improve upon the lower bounds previously established by Buyantogtokh, Horoldagva, and Das concerning  $GRM_\lambda$ of trees and unicyclic graphs of given order.
\end{abstract}


\medskip\noindent
\textbf{AMS Math.\ Subj.\ Class.\ (2020)}: 05C05, 05C07, 05C09, 05C35.

\section{Introduction}
Consider a simple graph $\Gamma$ where $V(\Gamma)$ denotes its vertex set and $E(\Gamma)$ denotes its edge set. For a vertex $a\in V (\Gamma)$, the open neighborhood $N_\Gamma(a)$ of $a$ in $\Gamma$ is the set $N_\Gamma(a)=\{b\in V(\Gamma)\mid ab\in E(\Gamma)\}$. The degree of $a$ in $\Gamma$, denoted as $\deg_\Gamma(a)$, is given by the order of its open neighborhood. Additionally, the distance between two vertices $a, b\in V(\Gamma)$, defined as the length of any shortest path in $\Gamma$ that connects $a$ and $b$, is denoted by $d_{\Gamma}(a,b)$. 

Topological indices are numerical descriptors derived from the molecular graph of a chemical compound, where atoms are represented as vertices and bonds as edges. These indices act as graph invariants, remaining unchanged under structural isomorphisms. They are widely used in chemical graph theory and quantitative structure-activity relationships (QSARs), linking the biological activity or properties of molecules to their chemical structures.
Vertex-degree-based topological indices are a specific category of topological indices that evaluate the characteristics of a graph by concentrating on the degrees of its vertices. These indices are defined using a set of real numbers that correspond to pairs of vertex degrees.

The first Zagreb index~\cite{zagreb1} and the second Zagreb index~\cite{zagreb2} are foundational members of the family of vertex-degree-based topological indices. These indices are defined for a graph $\Gamma$ as follows:
$$
M_1(\Gamma)=\sum_{a\in V(\Gamma)}\deg_{\Gamma}(a)^2
\qquad {\rm and} \qquad 
M_2(\Gamma)=\sum_{ab\in E(\Gamma)} \deg_{\Gamma}(a) \deg_{\Gamma}(b)\,.
$$
These indices have significant applications across multiple fields, such as chemistry and network analysis, where they aid in characterizing molecular structures and network topology. For a comprehensive and transparent overview of the Zagreb indices, see Ali et al.~\cite{akbar}, Borovi\'canin et al.~\cite{BDF}, and Gutman et al.~\cite{gutmil}. However, research in this area remains ongoing, with recent contributions from Ahmad et al.~\cite{ahmad-2023}, Lin and Qian~\cite{lin-2023}, Pirzada and Khan~\cite{pirzada-2023}, T\"{a}ubig~\cite{taubig-2023}, Yuan~\cite{yuan}, and Das et al.~\cite{ddm} providing new insights and advancements. 

Alongside the Zagreb indices, various other vertex-degree-based indices have been proposed. These include the atom-bond connectivity index~\cite{etrg}, the sum connectivity index~\cite{zhou}, irregularity indices~\cite{alb,add}, variable Zagreb indices~\cite{mm,mr}, multiplicative Zagreb indices ~\cite{todes,vb}, the Lanzhou index~\cite{DLLL,VLC}, and entire Zagreb indices~\cite{A,l}. 

Furtula et al.~\cite{FG} demonstrated that the difference $M_2(\Gamma)-M_1(\Gamma)$ is closely related to the reduced second Zagreb index $RM_2(\Gamma)$, which is defined as 
$$RM_2(\Gamma)=\sum_{ab\in E(\Gamma)}[(\deg_{\Gamma}(a)-1)(\deg_{\Gamma}(b)-1)].$$ 
This index has been studied in various contexts, including the works of Li et al.~\cite{lzz}, Buyantogtokh et al.~\cite{BD1}, Gao and Xu~\cite{gao-2020}, and Shafique and Ali~\cite{SHA}.

In 2019, Horoldagva et al.~\cite{HHH} extended the reduced second Zagreb index to the general reduced second Zagreb index $GRM_\lambda(\Gamma)$, defined as 
$$GRM_\lambda(\Gamma) = \sum_{ab\in E(\Gamma)}[(\deg_{\Gamma}(a)+\lambda)(\deg_{\Gamma}(b)+\lambda)]\,,$$
where $\lambda$ is an arbitrary, fixed real number. This definition can also be expressed equivalently as
$$GRM_\lambda(\Gamma) = M_2(\Gamma)+\lambda M_1(\Gamma)+\lambda^2|E(\Gamma)|.$$ 
This general version of the index has garnered considerable interest in recent research~\cite{HHH, BD, khoeilar-2021}, reflecting its significance in graph theory and its applications in chemical graph theory. The exploration of 
$GRM_\lambda(\Gamma)$ not only enhances our understanding of vertex-degree-based indices but also opens new avenues for analyzing the structural properties of graphs.

\subsection{Our primary motivation}

Our primary motivation for this paper arises from the following two results presented in~\cite{BD} on the general reduced second Zagreb index of trees and unicyclic graphs. 

\begin{prelem}\label{th1}
If $\lambda\ge -\frac{1}{2}$ and $T$ is a tree with $n$ vertices, then $$GRM_\lambda(T)\ge (2+\lambda)(n+2\lambda-1).$$
When $n=4$ and $\lambda=-\frac{1}{2}$, equality is achieved if and only if $T=P_4$ (a path graph with $4$ vertices) or $T=K_{1,3}$ (a star graph with $4$ vertices). For all other cases, equality is achieved if and only if $T=P_n$ (a path graph with $n$ vertices).
\end{prelem}

\begin{prelem}\label{th2}
If $\lambda\ge -\frac{1}{2}$ and $U$ is a unicyclic graph with $n$ vertices, then $$GRM_\lambda(U)\ge n\,(2+\lambda)^2,$$
where equality holds if and only if $T=C_n$ (a cycle graph with $n$ vertices).
\end{prelem}

\subsection{Our results}

In this paper, we extend and refine the bounds established in Theorem~\ref{th1} and Theorem~\ref{th2} by proving the following two theorems. Denoting by $\mathcal{T}_{n,\Delta}$ the set of all trees with $n$ vertices and a maximum degree of $\Delta$, the first results reads as follows. 

\begin{theorem}
\label{thm1}
If $\lambda\ge -\frac{1}{2}$, $n\ge 3$, and $T\in \mathcal{T}_{n,\Delta}$, then  
$$GRM_{\lambda}(T)\geqslant \left\{\begin{array}{lc}
 n(\lambda+2)^2-3(\lambda+2)+(\lambda+1)(\Delta^2-3\Delta-\lambda);& \Delta< n-1,\\
&\\
 (n-1)(n-1+\lambda)(1+\lambda);&
\Delta=n-1.
\end{array}\right.$$
The equality holds if and only if $T$ is a spider graph (a tree with at most one vertex
of a degree greater than two) with at most one leg of length more than one.
\end{theorem}

Let $\mathcal{U}_{n,\Delta}$ denote the set of all unicyclic graphs with $n$ vertices and maximum
degree ${\Delta}$. Let $\mathcal{U}^{(2)}_{n,\Delta}$ be a subset of $\mathcal{U}_{n,\Delta}$ which contains unicyclic graphs $U_2$ constructed as follows. $U_2$ is obtained from the disjoint union of a cycle $C$ and a star $K_{1,\Delta-1}$ by adding a path of length at least $2$ between the center $a$ of $K_{1,\Delta-1}$ and a vertex $b$ of $C$, as shown in Fig.~\ref{fig:U2}.

\begin{figure}[ht!] \centering
		\includegraphics[width=2.7in,height=1.1in]{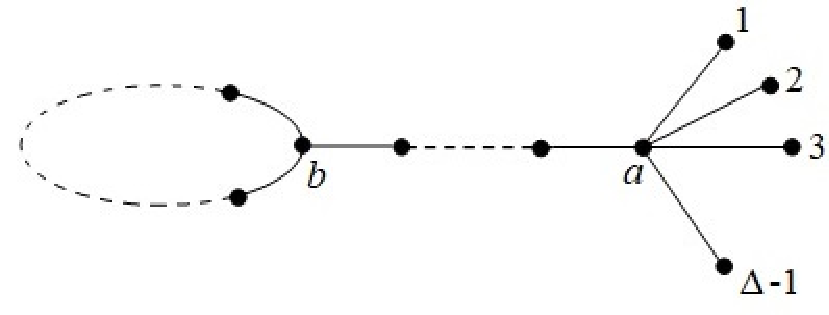}	
		\caption{The construction of unicyclic graphs from $\mathcal{U}^{(2)}_{n,\Delta}$.}
\label{fig:U2}
	\end{figure}

Let further ${U}^{(3)}_{n,\Delta}$ be a unicyclic graph from $\mathcal{U}_{n,\Delta}$ obtained by identifying a vertex $a$ of the cycle $C_{n-(\Delta-2)}$ and the center of a star $K_{1,\Delta-2}$, as shown in Fig.~\ref{fig:U3}.

\begin{figure}[ht!] \centering
		\includegraphics[width=1.6in,height=1in]{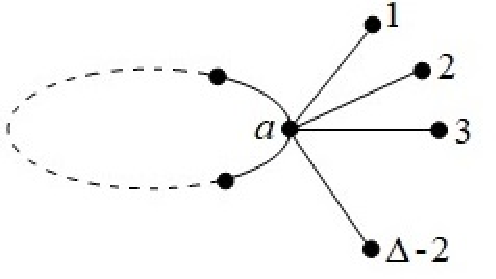}	
		\caption{The unicyclic graph ${U}^{(3)}_{n,\Delta}$.}
\label{fig:U3}
	\end{figure}

Our second main result reads as follows. 

\begin{theorem}
\label{thm3}
If $\lambda\ge -\frac{1}{2}$, $\Delta \geq 3$, and $U\in \mathcal{U}_{n,\Delta}$, then the following results hold:

(i) If $ 2\lambda+6<\Delta<n-3$, then 
\begin{align*}
GRM_{\lambda}(U)\ge (\Delta+\lambda)(\Delta+\Delta\lambda+1)+(n-\Delta)(2+\lambda)^2+3(2+\lambda),
\end{align*}
with equality if and only if $U\in \mathcal{U}^{(2)}_{n,\Delta}$.

(ii) If $3 \leq \Delta < 2\lambda+6$ or $n-3 \leq \Delta \leq n-1$, then
\begin{align*}
GRM_{\lambda}(U)\ge (\Delta+\lambda)(\Delta+\Delta\lambda+2)+(n-\Delta)(2+\lambda)^2,
\end{align*}
with equality if and only if $U= {U}^{(3)}_{n,\Delta}$.

(iii) If $\Delta = 2\lambda+6$, then 
\begin{align*}
GRM_{\lambda}(U)\ge  6(2+\lambda)^3+(n-2\lambda-6)(2+\lambda)^2,
\end{align*}
with equality if and only if either $U\in \mathcal{U}^{(2)}_{n,\Delta}$ or $U={U}^{(3)}_{n,\Delta}$.
\end{theorem}

\section{Proof of Theorem~\ref{thm1}}
\label{sec:trees}

In this section we prove Theorem~\ref{thm1} and  at the end of the section compare the bounds given in the theorem with those in Theorem~\ref{th1}.

A rooted tree is a tree in which one vertex is designated as the root. In this structure, every other vertex is connected to the root either directly or through a sequence of edges, creating a hierarchical organization among the vertices. 
A {\em spider} is a specific type of tree that has at most one vertex with a degree greater than two, known as the center of the spider. Each path extending from this center to a leaf vertex (vertex of degree one) is called a {\em leg}. A {\em star} is a special case of a spider where all legs have a length of one. By convention, a path graph can also be considered a spider, particularly when it is viewed as having either one leg (a single path) or two legs (two paths extending from a common vertex).

The key lemma for establishing the announced inequality on trees reads as follows.

\begin{lemma}\label{lem:key}
Let $\lambda\ge -\frac{1}{2}$ and $\Delta \ge 3$. If $T\in \mathcal{T}_{n,\Delta}$ contains two distinct vertices $a$ and $b$ such that $\deg_T(a)=\Delta$ and $\deg_T(b)\ge 3$, then there exists $T^{*}\in \mathcal{T}_{n,\Delta}$ such that $GRM_{\lambda}(T^{*})<GRM_{\lambda}(T)$.
\end{lemma}

\begin{proof}
Let $T$ be a rooted tree with root vertex $a$. Without loss of generality, assume $b$ is the vertex farthest from $a$ among all non root vertices $x$ (i.e., $x\ne a$) satisfying $\deg_T(x)\ge 3$. Let $\deg_{T}(b)=\ell$, and denote the neighborhood of $b$ as $N_T(b)=\{b_1, b_2,\ldots, b_{\ell}\}$, where $b_{\ell}$ is the unique neighbor of $b$ lying on the path from $b$ to $a$ in $T$. By our assumption about the maximality of $d_T(a,b)$, every vertex $b_i$ (for $1\leq i \leq \ell-1$) satisfies $\deg_{T}(b_i) \in \{1,2\}$. These degree constraints lead to three distinct cases:

\medskip\noindent
{\bf Case 1}: $b$ has at least two leaf neighbors. \\
Assume without loss of generality that the vertices $b_1$ and $b_2$ are two leaves adjacent to $b$. We construct a modified tree $T^{*}$ by removing the edge $bb_1$ and adding the edge $b_1b_2$. In $T^{*}$, vertex $b_2$ becomes the support vertex for leaf $b_1$ (see Fig.~\ref{fig:case1}). 

\begin{figure}[ht!]
\begin{tikzpicture}[line cap=round,line join=round,>=triangle 45,x=1.0cm,y=1.0cm]
\clip(-3.5,0.08) rectangle (9.6,5.64);
\draw (2,5)-- (1.2,4.2);
\draw (2,5)-- (1.96,4.2);
\draw (2,5)-- (3.02,4.2);
\draw [dash pattern=on 2pt off 2pt] (1.2,4.2)-- (1.2,3.38);
\draw (1.2,3.38)-- (1.2,2.58);
\draw (1.2,2.58)-- (0.62,1.7);
\draw (1.2,2.58)-- (1.3,1.7);
\draw (1.2,2.58)-- (1.84,1.7);
\draw (1.2,2.58)-- (2.98,1.7);
\draw [->] (3.74,2.74) -- (4.8,2.74);
\draw (7,5)-- (6.38,4.2);
\draw (7,5)-- (7.16,4.2);
\draw (7,5)-- (8.28,4.16);
\draw [dash pattern=on 2pt off 2pt] (6.38,4.2)-- (6.38,3.38);
\draw (6.38,3.38)-- (6.38,2.58);
\draw (6.38,2.58)-- (5.74,1.8);
\draw (6.38,2.58)-- (6.42,1.8);
\draw (6.38,2.58)-- (7.72,1.8);
\draw (5.74,1.8)-- (5.72,0.9);
\draw [rotate around={-0.54:(7.91,4.17)}] (7.91,4.17) ellipse (1.13cm and 0.37cm);
\draw [rotate around={-1.12:(2.62,4.2)}] (2.62,4.2) ellipse (1.09cm and 0.39cm);
\draw [rotate around={-1.22:(2.48,1.6)}] (2.48,1.6) ellipse (1.02cm and 0.4cm);
\draw [rotate around={-2.22:(7.4,1.7)}] (7.4,1.7) ellipse (1.2cm and 0.4cm);
\fill [color=black] (2,5) circle (1.5pt);
\draw[color=black] (2,5.3) node {$a$};
\fill [color=black] (1.2,4.2) circle (1.5pt);
\draw[color=black] (0.9,4.02) node {$a_1$};
\fill [color=black] (1.96,4.2) circle (1.5pt);
\draw[color=black] (2.3,4.1) node {$a_2$};
\fill [color=black] (2.48,4.2) circle (0.5pt);
\fill [color=black] (2.62,4.2) circle (0.5pt);
\fill [color=black] (2.76,4.2) circle (0.5pt);
\fill [color=black] (3.02,4.2) circle (1.5pt);
\draw[color=black] (3.3,4.1) node {$a_{\Delta}$};
\fill [color=black] (1.2,3.38) circle (1.5pt);
\draw[color=black] (0.9,3.42) node {$b_{\ell}$};
\fill [color=black] (1.2,2.58) circle (1.5pt);
\draw[color=black] (1,2.7) node {$b$};
\fill [color=black] (0.62,1.7) circle (1.5pt);
\draw[color=black] (0.56,1.48) node {$b_1$};
\fill [color=black] (1.3,1.7) circle (1.5pt);
\draw[color=black] (1.28,1.46) node {$b_2$};
\fill [color=black] (1.84,1.7) circle (1.5pt);
\draw[color=black] (1.86,1.45) node {$b_3$};
\fill [color=black] (2.24,1.7) circle (0.5pt);
\fill [color=black] (2.4,1.7) circle (0.5pt);
\fill [color=black] (2.52,1.7) circle (0.5pt);
\fill [color=black] (2.98,1.7) circle (1.5pt);
\draw[color=black] (3.04,1.48) node {$b_{\ell-1}$};
\draw[color=black] (1.68,0.86) node {$T$};
\fill [color=black] (7,5) circle (1.5pt);
\draw[color=black] (7.04,5.28) node {$a$};
\fill [color=black] (6.38,4.2) circle (1.5pt);
\draw[color=black] (6.1,4.06) node {$a_1$};
\fill [color=black] (7.16,4.2) circle (1.5pt);
\draw[color=black] (7.4,4) node {$a_2$};
\fill [color=black] (7.54,4.2) circle (0.5pt);
\fill [color=black] (7.7,4.2) circle (0.5pt);
\fill [color=black] (7.82,4.2) circle (0.5pt);
\fill [color=black] (8.28,4.16) circle (1.5pt);
\draw[color=black] (8.6,4.1) node {$a_{\Delta}$};
\fill [color=black] (6.38,3.38) circle (1.5pt);
\draw[color=black] (6.1,3.4) node {$b_{\ell}$};
\fill [color=black] (6.38,2.58) circle (1.5pt);
\draw[color=black] (6.2,2.6) node {$b$};
\fill [color=black] (5.74,1.8) circle (1.5pt);
\draw[color=black] (6,1.68) node {$b_2$};
\fill [color=black] (6.42,1.8) circle (1.5pt);
\draw[color=black] (6.62,1.58) node {$b_3$};
\fill [color=black] (6.9,1.8) circle (0.5pt);
\fill [color=black] (7.06,1.8) circle (0.5pt);
\fill [color=black] (7.2,1.8) circle (0.5pt);
\fill [color=black] (7.72,1.8) circle (1.5pt);
\draw[color=black] (7.96,1.58) node {$b_{\ell-1}$};
\fill [color=black] (5.72,0.9) circle (1.5pt);
\draw[color=black] (6.06,0.9) node {$b_1$};
\draw[color=black] (7.18,0.7) node {$T^{*}$};
\end{tikzpicture}
\caption{Transformation from Case 1 of Lemma \ref{lem:key}.}
\label{fig:case1}
\end{figure}
We now analyze how this modification affects the general reduced second Zagreb index. Recall that $\lambda\ge -\frac{1}{2}$ and $\ell\ge 3$. We define $X = GRM_{\lambda}(T)-GRM_{\lambda}(T^{*})$ and calculate it as follows: 
\begin{align*}
X & = (\deg_{T}(b)+\lambda)(\deg_{T}(b_1)+\lambda)+(\deg_{T}(b)+\lambda)(\deg_{T}(b_2)+\lambda)\\
&\quad + (\deg_{T}(b)+\lambda)(\deg_{T}(b_\ell)+\lambda)+\sum_{i=3}^{\ell-1}(\deg_{T}(b)+\lambda)(\deg_{T}(b_i)+\lambda)\\
&\quad - (\deg_{T^{*}}(b_1)+\lambda)(\deg_{T^{*}}(b_2)+\lambda)-(\deg_{T^{*}}(b)+\lambda)(\deg_{T^{*}}(b_2)+\lambda)\\
&\quad -(\deg_{T^{*}}(b)+\lambda)(\deg_{T^{*}}(b_\ell)+\lambda)-\sum_{i=3}^{\ell-1}(\deg_{T^{*}}(b)+\lambda)(\deg_{T^{*}}(b_i)+\lambda)\\
& = 2(1+\lambda)(\ell+\lambda)+(\ell+\lambda)(\deg_{T}(b_\ell)+\lambda)\\
&\quad + \sum_{i=3}^{\ell-1}(\ell+\lambda)(\deg_{T}(b_i)+\lambda)\\
&\quad - (1+\lambda)(2+\lambda)-(2+\lambda)(\ell-1+\lambda)-(\ell-1+\lambda)(\deg_{T}(b_\ell)+\lambda)\\
&\quad - \sum_{i=3}^{\ell-1}(\ell-1+\lambda)(\deg_{T}(b_i)+\lambda)\\
& = \lambda\ell-\lambda+\deg_{T}(b_\ell)+\sum_{i=3}^{\ell-1}(\deg_{T}(b_i)+\lambda)\\
& \ge \lambda(\ell-1)+2+(1+\lambda)(\ell-3)
\\& = 2\lambda\ell-4\lambda+\ell-1
\\& = 2\lambda(\ell-2)+(\ell-2)+1
\\& = (1+2\lambda)(\ell-2)+1
>0.
\end{align*}

\medskip\noindent
{\bf Case 2}: $b$ has exactly one leaf neighbor.\\
Let $b_1$ be the leaf adjacent to $b$. Consider the path $bc_1c_2\ldots c_k$, where $c_1=b_2$, $c_k$ is a leaf, and $k\ge 2$.  
Construct a modified tree $T^{*}$ by removing the leaf $b_1$, detaching the subpath $c_1c_2\ldots c_k$ from $b$, and attaching the extended path $c_1 c_2\ldots c_k b_1$ to $b$. This reorganization creates a new leaf $b_1$ at the path's terminus (see Fig.~\ref{fig:case2}). 

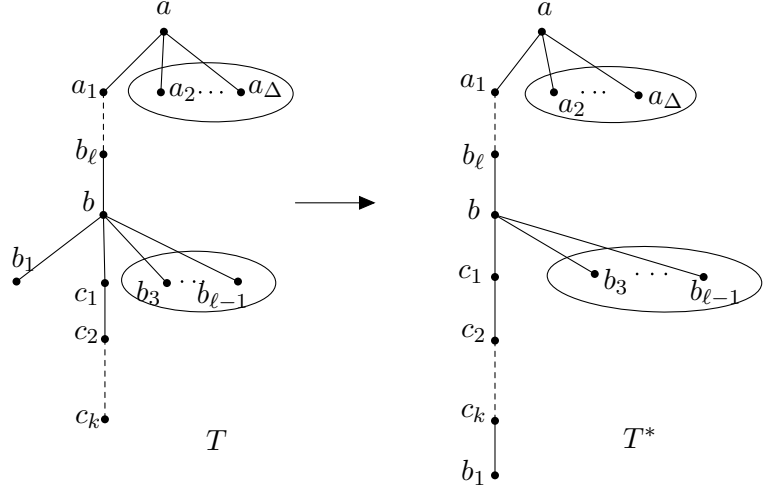
\begin{figure}[ht!]
\begin{tikzpicture}[line cap=round,line join=round,>=triangle 45,x=1.0cm,y=1.0cm]
\clip(-3.5,-1.52) rectangle (10.28,5.66);
\draw (2,5)-- (1.2,4.2);
\draw (2,5)-- (1.96,4.2);
\draw (2,5)-- (3.02,4.2);
\draw [dash pattern=on 2pt off 2pt] (1.2,4.2)-- (1.2,3.38);
\draw (1.2,3.38)-- (1.2,2.58);
\draw (1.2,2.58)-- (0.05,1.7);
\draw (1.2,2.58)-- (1.22,1.68);
\draw (1.2,2.58)-- (2.04,1.68);
\draw (1.2,2.58)-- (2.98,1.7);
\draw [->] (3.74,2.74) -- (4.8,2.74);
\draw (7,5)-- (6.38,4.2);
\draw (7,5)-- (7.16,4.2);
\draw (7,5)-- (8.28,4.16);
\draw [dash pattern=on 2pt off 2pt] (6.38,4.2)-- (6.38,3.38);
\draw (6.38,3.38)-- (6.38,2.58);
\draw (6.38,2.58)-- (6.38,1.76);
\draw (6.38,2.58)-- (7.7,1.8);
\draw [rotate around={-0.54:(7.91,4.17)}] (7.91,4.17) ellipse (1.13cm and 0.37cm);
\draw [rotate around={-1.12:(2.62,4.2)}] (2.62,4.2) ellipse (1.09cm and 0.39cm);
\draw (1.22,1.68)-- (1.22,0.94);
\draw [dash pattern=on 2pt off 2pt] (1.22,0.94)-- (1.22,-0.12);
\draw [rotate around={-1.22:(2.46,1.7)}] (2.46,1.7) ellipse (1.02cm and 0.4cm);
\draw (6.38,2.58)-- (9.14,1.76);
\draw (6.38,1.76)-- (6.38,0.92);
\draw [dash pattern=on 2pt off 2pt] (6.38,0.92)-- (6.38,-0.14);
\draw (6.38,-0.14)-- (6.38,-0.86);
\draw [rotate around={-2.12:(8.49,1.73)}] (8.49,1.73) ellipse (1.42cm and 0.43cm);
\fill [color=black] (2,5) circle (1.5pt);
\draw[color=black] (2,5.32) node {$a$};
\fill [color=black] (1.2,4.2) circle (1.5pt);
\draw[color=black] (0.96,4.24) node {$a_1$};
\fill [color=black] (1.96,4.2) circle (1.5pt);
\draw[color=black] (2.25,4.2) node {$a_2$};
\fill [color=black] (2.48,4.2) circle (0.5pt);
\fill [color=black] (2.62,4.2) circle (0.5pt);
\fill [color=black] (2.76,4.2) circle (0.5pt);
\fill [color=black] (3.02,4.2) circle (1.5pt);
\draw[color=black] (3.35,4.25) node {$a_{\Delta}$};
\fill [color=black] (1.2,3.38) circle (1.5pt);
\draw[color=black] (0.98,3.42) node {$b_\ell$};
\fill [color=black] (1.2,2.58) circle (1.5pt);
\draw[color=black] (1,2.76) node {$b$};
\fill [color=black] (0.05,1.7) circle (1.5pt);
\draw[color=black] (0.15,2) node {$b_1$};
\fill [color=black] (1.22,1.68) circle (1.5pt);
\draw[color=black] (0.98,1.5) node {$c_1$};
\fill [color=black] (2.04,1.68) circle (1.5pt);
\draw[color=black] (1.8,1.55) node {$b_3$};
\fill [color=black] (2.24,1.7) circle (0.5pt);
\fill [color=black] (2.4,1.7) circle (0.5pt);
\fill [color=black] (2.52,1.7) circle (0.5pt);
\fill [color=black] (2.98,1.7) circle (1.5pt);
\draw[color=black] (2.78,1.5) node {$b_{\ell-1}$};
\fill [color=black] (1.22,0.94) circle (1.5pt);
\draw[color=black] (0.98,1) node {$c_2$};
\fill [color=black] (7,5) circle (1.5pt);
\draw[color=black] (7.04,5.3) node {$a$};
\fill [color=black] (6.38,4.2) circle (1.5pt);
\draw[color=black] (6.1,4.33) node {$a_1$};
\fill [color=black] (7.16,4.2) circle (1.5pt);
\draw[color=black] (7.4,4) node {$a_2$};
\fill [color=black] (7.54,4.2) circle (0.5pt);
\fill [color=black] (7.7,4.2) circle (0.5pt);
\fill [color=black] (7.82,4.2) circle (0.5pt);
\fill [color=black] (8.28,4.16) circle (1.5pt);
\draw[color=black] (8.64,4.1) node {$a_{\Delta}$};
\fill [color=black] (6.38,3.38) circle (1.5pt);
\draw[color=black] (6.1,3.35) node {$b_\ell$};
\fill [color=black] (6.38,2.58) circle (1.5pt);
\draw[color=black] (6.1,2.66) node {$b$};
\fill [color=black] (6.38,1.76) circle (1.5pt);
\draw[color=black] (6.06,1.8) node {$c_1$};
\fill [color=black] (8.26,1.8) circle (0.5pt);
\fill [color=black] (8.46,1.8) circle (0.5pt);
\fill [color=black] (8.64,1.8) circle (0.5pt);
\fill [color=black] (7.7,1.8) circle (1.5pt);
\draw[color=black] (7.98,1.7) node {$b_3$};
\fill [color=black] (6.38,0.92) circle (1.5pt);
\draw[color=black] (6.1,1.02) node {$c_2$};
\fill [color=black] (1.22,-0.12) circle (1.5pt);
\draw[color=black] (1,-0.12) node {$c_k$};
\fill [color=black] (9.14,1.76) circle (1.5pt);
\draw[color=black] (9.3,1.55) node {$b_{\ell-1}$};
\fill [color=black] (6.38,-0.14) circle (1.5pt);
\draw[color=black] (6.1,-0.04) node {$c_k$};
\fill [color=black] (6.38,-0.86) circle (1.5pt);
\draw[color=black] (6.1,-0.82) node {$b_1$};
\draw[color=black] (2.7,-0.4) node {$T$};
\draw[color=black] (8.3,-0.36) node {$T^{*}$};
\end{tikzpicture}
\caption{The construction from Case 2 of Lemma \ref{lem:key}.}
\label{fig:case2}
\end{figure}

Using the fact that $\lambda\ge -\frac{1}{2}$ and $\ell\ge 3$, and defining $X = GRM_{\lambda}(T)-GRM_{\lambda}(T^{*})$, we have:
\begin{align*}
X & = (\deg_{T}(b)+\lambda)(\deg_{T}(b_1)+\lambda)+(\deg_{T}(c_k)+\lambda)(\deg_{T}(c_{k-1})+\lambda)\\
& \quad + (\deg_{T}(b)+\lambda)(\deg_{T}(b_\ell)+\lambda)+\sum_{i=2}^{\ell-1}(\deg_{T}(b)+\lambda)(\deg_{T}(b_i)+\lambda)\\
& \quad -(\deg_{T^{*}}(b_1)+\lambda)(\deg_{T^{*}}(c_k)+\lambda)-(\deg_{T^{*}}(c_k)+\lambda)(\deg_{T^{*}}(c_{k-1})+\lambda)\\
& \quad -(\deg_{T^{*}}(b)+\lambda)(\deg_{T^{*}}(b_\ell)+\lambda)-\sum_{i=2}^{\ell-1}(\deg_{T^{*}}(b)+\lambda)(\deg_{T^{*}}(b_i)+\lambda)\\
& = (1+\lambda)(\ell+\lambda)+(1+\lambda)(2+\lambda)+(\ell+\lambda)(\deg_{T}(b_\ell)+\lambda)\\
& \quad + \sum_{i=2}^{\ell-1}(\ell+\lambda)(\deg_{T}(b_i)+\lambda)\\
& \quad -(1+\lambda)(2+\lambda)-(2+\lambda)^2-(\ell-1+\lambda)(\deg_{T}(b_\ell)+\lambda)\\
& \quad - \sum_{i=2}^{\ell-1}(\ell-1+\lambda)(\deg_{T}(b_i)+\lambda)\\
&=(1+\lambda)(\ell-2)-(2+\lambda)+\sum_{i=2}^{\ell}(\deg_{T}(b_i)+\lambda)\\
& \geq (1+\lambda)(\ell-2)-(2+\lambda)+(\ell-1)(2+\lambda)
\\& = (3+2\lambda)(\ell-2)
>0.
\end{align*}

\medskip\noindent
{\bf Case 3}: $b$ has no leaf neighbors. \\
Consider two paths $bc_1c_2\ldots c_k$ and $bd_1d_2\ldots d_s$ with $c_1=b_1$ and $d_1=b_2$, where both paths have length at lease two (i.e., $k,s\ge 2$) and terminal nodes $c_k$ and $d_s$ are leaves. Construct a modified tree $T^{*}$ by detaching the path $c_1c_2\ldots c_k$ from $b$ and attaching it to $d_s$ (see Fig.~\ref{fig:case3}). 

\begin{figure}[ht!]
\begin{tikzpicture}[line cap=round,line join=round,>=triangle 45,x=0.8cm,y=1.0cm]
\clip(-5,-3.4) rectangle (17.58,6.3);
\draw (2,5)-- (1.2,4.2);
\draw (2,5)-- (1.96,4.2);
\draw (2,5)-- (3.02,4.2);
\draw [dash pattern=on 2pt off 2pt] (1.2,4.2)-- (1.2,3.38);
\draw (1.2,3.38)-- (1.2,2.58);
\draw (1.2,2.58)-- (-0.04,1.68);
\draw (1.2,2.58)-- (1.22,1.68);
\draw (1.2,2.58)-- (2.04,1.68);
\draw (1.2,2.58)-- (2.98,1.7);
\draw [->] (3.74,2.74) -- (4.8,2.74);
\draw (7,5)-- (6.38,4.2);
\draw (7,5)-- (7.16,4.2);
\draw (7,5)-- (8.28,4.16);
\draw [dash pattern=on 2pt off 2pt] (6.38,4.2)-- (6.38,3.38);
\draw (6.38,3.38)-- (6.38,2.58);
\draw (6.38,2.58)-- (6.38,1.76);
\draw (6.38,2.58)-- (7.7,1.8);
\draw [rotate around={-0.54:(7.91,4.17)}] (7.91,4.17) ellipse (1.13cm and 0.37cm);
\draw [rotate around={-1.12:(2.62,4.2)}] (2.62,4.2) ellipse (1.09cm and 0.39cm);
\draw (1.22,1.68)-- (1.22,0.94);
\draw [dash pattern=on 2pt off 2pt] (1.22,0.94)-- (1.22,-0.01);
\draw [rotate around={-1.22:(2.46,1.7)}] (2.46,1.7) ellipse (1.02cm and 0.4cm);
\draw (6.38,2.58)-- (9.14,1.76);
\draw (6.38,1.76)-- (6.38,0.92);
\draw [dash pattern=on 2pt off 2pt] (6.38,0.92)-- (6.38,-0.14);
\draw (6.38,-0.14)-- (6.38,-0.86);
\draw [rotate around={-2.12:(8.49,1.73)}] (8.49,1.73) ellipse (1.42cm and 0.43cm);
\draw (-0.04,1.68)-- (-0.04,0.86);
\draw [dash pattern=on 2pt off 2pt] (-0.04,0.86)-- (-0.04,-0.04);
\draw (6.38,-0.86)-- (6.38,-1.54);
\draw [dash pattern=on 2pt off 2pt] (6.38,-1.54)-- (6.38,-2.62);
\fill [color=black] (2,5) circle (1.5pt);
\draw[color=black] (2,5.32) node {$a$};
\fill [color=black] (1.2,4.2) circle (1.5pt);
\draw[color=black] (0.96,4.24) node {$a_1$};
\fill [color=black] (1.96,4.2) circle (1.5pt);
\draw[color=black] (2.16,4) node {$a_2$};
\fill [color=black] (2.48,4.2) circle (0.5pt);
\fill [color=black] (2.62,4.2) circle (0.5pt);
\fill [color=black] (2.76,4.2) circle (0.5pt);
\fill [color=black] (3.02,4.2) circle (1.5pt);
\draw[color=black] (3.35,4.1) node {$a_{\Delta}$};
\fill [color=black] (1.2,3.38) circle (1.5pt);
\draw[color=black] (0.98,3.42) node {$b_\ell$};
\fill [color=black] (1.2,2.58) circle (1.5pt);
\draw[color=black] (1,2.76) node {$b$};
\fill [color=black] (-0.04,1.68) circle (1.5pt);
\draw[color=black] (-0.28,1.7) node {$c_1$};
\fill [color=black] (-0.04,0.86) circle (1.5pt);
\draw[color=black] (-0.28,0.96) node {$c_2$};
\fill [color=black] (-0.04,-0.04) circle (1.5pt);
\draw[color=black] (-0.28,0.1) node {$c_k$};
\fill [color=black] (1.22,1.68) circle (1.5pt);
\draw[color=black] (1,1.65) node {$d_1$};
\fill [color=black] (1.22,0.94) circle (1.5pt);
\draw[color=black] (1,1.04) node {$d_2$};
\fill [color=black] (1.22,-0.01) circle (1.5pt);
\draw[color=black] (1,-0.02) node {$d_s$};
\fill [color=black] (2.04,1.68) circle (1.5pt);
\draw[color=black] (1.9,1.5) node {$b_3$};
\fill [color=black] (2.24,1.7) circle (0.5pt);
\fill [color=black] (2.4,1.7) circle (0.5pt);
\fill [color=black] (2.52,1.7) circle (0.5pt);
\fill [color=black] (2.98,1.7) circle (1.5pt);
\draw[color=black] (2.75,1.5) node {$b_{\ell-1}$};
\fill [color=black] (7,5) circle (1.5pt);
\draw[color=black] (7.04,5.3) node {$a$};
\fill [color=black] (6.38,4.2) circle (1.5pt);
\draw[color=black] (6.1,4.3) node {$a_1$};
\fill [color=black] (7.16,4.2) circle (1.5pt);
\draw[color=black] (7.36,4) node {$a_2$};
\fill [color=black] (7.54,4.2) circle (0.5pt);
\fill [color=black] (7.7,4.2) circle (0.5pt);
\fill [color=black] (7.82,4.2) circle (0.5pt);
\fill [color=black] (8.28,4.16) circle (1.5pt);
\draw[color=black] (8.64,4.1) node {$a_{\Delta}$};
\fill [color=black] (6.38,3.38) circle (1.5pt);
\draw[color=black] (6.1,3.48) node {$b_\ell$};
\fill [color=black] (6.38,2.58) circle (1.5pt);
\draw[color=black] (6.1,2.74) node {$b$};
\fill [color=black] (6.38,1.76) circle (1.5pt);
\draw[color=black] (6.1,1.8) node {$d_1$};
\fill [color=black] (8.26,1.8) circle (0.5pt);
\fill [color=black] (8.46,1.8) circle (0.5pt);
\fill [color=black] (8.64,1.8) circle (0.5pt);
\fill [color=black] (7.7,1.8) circle (1.5pt);
\draw[color=black] (7.98,1.7) node {$b_3$};
\fill [color=black] (6.38,0.92) circle (1.5pt);
\draw[color=black] (6.1,1.02) node {$d_2$};
\fill [color=black] (9.14,1.76) circle (1.5pt);
\draw[color=black] (9.24,1.6) node {$b_{\ell-1}$};
\fill [color=black] (6.38,-0.14) circle (1.5pt);
\draw[color=black] (6.1,-0.04) node {$d_s$};
\fill [color=black] (6.38,-0.86) circle (1.5pt);
\draw[color=black] (6.1,-0.84) node {$c_1$};
\draw[color=black] (2.34,-0.66) node {$T$};
\draw[color=black] (8.7,-1.16) node {$T^{*}$};
\fill [color=black] (6.38,-1.54) circle (1.5pt);
\draw[color=black] (6.1,-1.5) node {$c_2$};
\fill [color=black] (6.38,-2.62) circle (1.5pt);
\draw[color=black] (6.1,-2.34) node {$c_k$};
\end{tikzpicture}
\caption{Transformation from Case 3 of Lemma \ref{lem:key}.}
\label{fig:case3}
\end{figure}
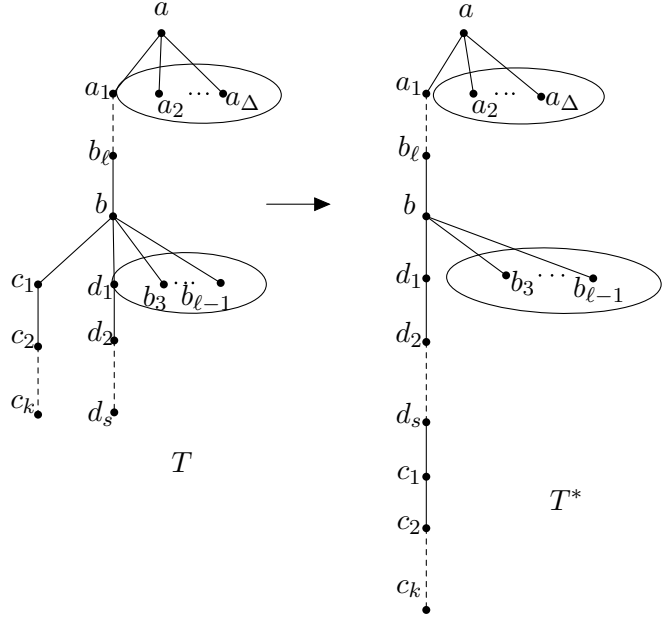

Once again, let $X = GRM_{\lambda}(T)- GRM_{\lambda}(T^{*})$. We can estimate it as follows:
\begin{align*}
X & = (\deg_{T}(b)+\lambda)(\deg_{T}(b_1)+\lambda)+(\deg_{T}(d_s)+\lambda)(\deg_{T}(d_{s-1})+\lambda) \\
&\quad + (\deg_{T}(b)+\lambda)(\deg_{T}(b_\ell)+\lambda)+\sum_{i=2}^{\ell-1}(\deg_{T}(b)+\lambda)(\deg_{T}(b_i)+\lambda) \\
&\quad - (\deg_{T^{*}}(b_1)+\lambda)(\deg_{T^{*}}(d_s)+\lambda)-(\deg_{T^{*}}(d_s)+\lambda)(\deg_{T^{*}}(d_{s-1})+\lambda) \\
&\quad - (\deg_{T^{*}}(b)+\lambda)(\deg_{T^{*}}(b_\ell)+\lambda)-\sum_{i=2}^{\ell-1}(\deg_{T^{*}}(b)+\lambda)(\deg_{T^{*}}(b_i)+\lambda) \\
& = (\ell+\lambda)(2+\lambda)+(1+\lambda)(2+\lambda)+(\ell+\lambda)(\deg_{T}(b_\ell)+\lambda) \\
&\quad + \sum_{i=2}^{\ell-1}(\ell+\lambda)(\deg_{T}(b_i)+\lambda) \\
&\quad - (2+\lambda)^2-(2+\lambda)^2-(\ell-1+\lambda)(\deg_{T}(b_\ell)+\lambda) \\
&\quad - \sum_{i=2}^{\ell-1}(\ell-1+\lambda)(\deg_{T}(b_i)+\lambda) \\
& = (2+\lambda)(\ell-3)+\sum_{i=2}^{\ell}(\deg_{T}(b_i)+\lambda)\\ 
& \geq (2+\lambda)(\ell-3)+(2+\lambda)(\ell-1)=2(2+\lambda)(\ell-2)
> 0.
\end{align*}
This concludes the demonstration of Lemma~\ref{lem:key}.
\end{proof}

Our second lemma addresses spiders. 

\begin{lemma}
\label{lem:spiders} 
If $T$ is a spider graph in $\mathcal{T}_{n,\Delta}$ with $\Delta\ge 3$, and it has at least two legs of length greater than one, then there exists another 
spider graph $T^{*}$ in $\mathcal{T}_{n,\Delta}$ such that $GRM_{\lambda}(T^{*})<GRM_{\lambda}(T)$.
\end{lemma}

\begin{proof}
Let $a$ be the center of $T$, and consider two legs of length greater than one represented by the paths $a b_1b_2\ldots b_t$ and $ac_1c_2\ldots c_k$. Construct $T^{*}$ by detaching the subpath $b_2\ldots b_t$ from $a$ and attaching it to the terminal vertex $c_k$. Define $X$ as $X= GRM_{\lambda}(T)-GRM_{\lambda}(T^{*})$. Then 
\begin{align*}
X & = (\deg_{T}(a)+\lambda)(\deg_{T}(b_1)+\lambda)+(\deg_{T}(b_1)+\lambda)(\deg_{T}(b_2)+\lambda)\\
&\quad + (\deg_{T}(c_k)+\lambda)(\deg_{T}(c_{k-1})+\lambda)\\
&\quad - (\deg_{T^{*}}(a)+\lambda)(\deg_{T^{*}}(b_1)+\lambda)-(\deg_{T^{*}}(c_k)+\lambda)(\deg_{T^{*}}(c_{k-1})+\lambda)\\
&\quad - (\deg_{T^{*}}(b_2)+\lambda)(\deg_{T^{*}}(c_k)+\lambda)\\
& = (2+\lambda)(\Delta+\lambda)+(2+\lambda)(\deg_{T}(b_2)+\lambda)+(1+\lambda)(2+\lambda)\\
&\quad - (1+\lambda)(\Delta+\lambda)-(\deg_{T^*}(b_2)+\lambda)(2+\lambda)-(2+\lambda)^2\\
& = \Delta-2>0,
\end{align*}
and we are done. 
\end{proof}

We now proceed to establish the proof of Theorem~\ref{thm1}. 

Let $T^*$ be a tree in $\mathcal{T}_{n,\Delta}$ such that $GRM_{\lambda}(T^*)\le GRM_{\lambda}(T)$ for all $T$ in $\mathcal{T}_{n,\Delta}$. Suppose $T^*$ is rooted at $a$ with $\deg_{T^*}(a)=\Delta$. If
$\Delta=2$, then $T^*$ is a path of order $n\ge 3$, and its general reduced second Zagreb index is given by: $$GRM_{\lambda}(T^*)=GRM_{\lambda}(P_n)=(n-3)(2+\lambda)^2+2(1+\lambda)(2+\lambda).$$ 
Given $\Delta\ge 3$, Lemma \ref{lem:key} ensures that $T^*$ is a
spider graph centered at vertex $a$. Furthermore, Lemma~\ref{lem:spiders} guarantees that $T^*$ has at most one leg of length exceeding one. If all legs of $T^*$ have length one, then $T^*$ is a star, $\Delta=n-1$, and its general reduced second Zagreb index is given by: $$GRM_{\lambda}(T^*)=GRM_{\lambda}(S_n)=(n-1)(n-1+\lambda)(1+\lambda).$$ Assume $T^*$ is not a star and has exactly one leg of length exceeding one. Then $\Delta<n-1$ and we have 
\begin{align*}
GRM_{\lambda}(T^*) = & (n-\Delta-2)(2+\lambda)^2+(\Delta-1)(\Delta+\lambda)(1+\lambda)\\
&+(\Delta+\lambda)(2+\lambda)+(1+\lambda)(2+\lambda)\\
=& n(2+\lambda)^2-3(2+\lambda)+(1+\lambda)(\Delta^2-3\Delta-\lambda),
\end{align*} 
from which the desired result follows.

\medskip
To end the section, we compare the bounds given in Theorem~\ref{th1} and Theorem~\ref{thm1}. At first consider the case when $\lambda\ge -\frac{1}{2}$ and $\Delta=n-1$. If $n=3$, then $T=P_3$ and both bounds yield the exact value of $GRM_{\lambda}(P_3)$ that is $2(1+\lambda)(2+\lambda)$. If $n\ge 4$, then
\begin{align*}
&\ (n-1)(n-1+\lambda)(1+\lambda)-
(2+\lambda)(n+2\lambda-1)\\
= &\ (n-3)\big(\lambda^2+n(\lambda+1)+(\lambda-1)\big) \\
 \geq  &\ (4-3)\big(0+4(-\frac{1}{2}+1)-\frac{1}{2}-1\big)=\frac{1}{2}>0.
\end{align*}

If $\lambda \ge - \frac{1}{2}$ and $\Delta<n-1$, then $n\ge 4$ and we have: 
\begin{align*}
&\ n(2+\lambda)^2-3(2+\lambda)+(1+\lambda)(\Delta^2-3\Delta-\lambda)-(2+\lambda)(n+2\lambda-1)\\
= &\ (1+\lambda)\big((\Delta-\frac{3}{2})^2+(n-3)\lambda+2n-\frac{25}{4}\big) \\
\geq &\ (-\frac{1}{2}+1)\big((2-\frac{3}{2})^2-\frac{1}{2}(n-3)+2n-\frac{25}{4}\big)\\
=&\ \frac{3}{4}(n-3)>0.
\end{align*} 
The calculations show that the lower bound established in Theorem~\ref{thm1} is stronger than the one presented in Theorem~\ref{th1}.

\section{Proof of Theorem~\ref{thm3}}

In this section we prove Theorem~\ref{thm3} and  provide a comparison between the bound of our theorem with the bound of Theorem~\ref{th2}.

Let $\mathcal{U}_{n,\Delta}$ represent the set of all unicyclic graphs with $n$ vertices and maximum
degree ${\Delta}$. If $U$ is a graph in $\mathcal{U}_{n,2}$, then $U$ must be the cycle graph $C_n$, for which the general reduced second Zagreb index is given by: $$GRM_{\lambda}(U)=n(2+\lambda)^2.$$ 
For the remainder of this section, we assume $\Delta \ge 3$. To establish the announced inequality for unicyclic graphs, we first need to prove several key lemmas.

\begin{lemma}\label{L4}
Assume $\lambda\ge -\frac{1}{2}$, and let $U$ be a unicyclic graph in $\mathcal{U}_{n,\Delta}$, where no vertex of degree ${\Delta}$ lies on the cycle. 
Suppose $\deg_{U}(a)=\Delta$, and let $b$ be a vertex 
on the cycle of $U$ that minimizes the distance $d_{U}(a,b)$. If either $\deg_{U}(b)\ge4$ or there exists a vertex $c$ (distinct from $a$ and $b$) with $\deg_{U}(c)\ge3$, then $\mathcal{U}_{n,\Delta}$ contains another unicyclic graph $U'$ such that $GRM_{\lambda}(U)>GRM_{\lambda}(U')$.
\end{lemma}
\begin{proof}
Let $C$ denote the cycle of $U$, and let $P$ be the path connecting $a$ to $b$. 
Consider a vertex $c$ in $U$, distinct from $a$ and $b$, such that $\deg_{U}(c)\ge3$. 
Define $T_c$ as the rooted tree with $c$ as its root, which maximizes the number of vertices connected to $c$. 

If $c$ is not a vertex of the cycle $C$ and the path $P$, then by Lemma \ref{lem:key}, we can transform $T_c$ into a path $P_c$ with the same number of vertices, such that $GRM_{\lambda}(T_{c})>GRM_{\lambda}(P_{c})$. Construct $U'$ in $\mathcal{U}_{n,\Delta}$ by removing $T_{c}$ from $U$ and replacing it with $P_c$. 
Consequently, it follows that $GRM_{\lambda}(U)>GRM_{\lambda}(U')$. 

Now, let $c$ be a vertex lying either on the cycle $C$ or on the path $P$. We consider the case where $c$ is a vertex on the cycle $C$; the case in which $c$ lies on the path $P$ can be proved analogously. Let $c_1$ and $c_2$ denote the neighbors of $c$ on $C$, excluding any vertices that belong to $T_c$. By Lemma \ref{lem:key}, we can transform $T_c$ into a path $P_c$ with the same number of vertices, such that $GRM_{\lambda}(T_c)\ge GRM_{\lambda}(P_c)$. Let $P_c=\gamma_1 \gamma_2 \cdots \gamma_k$, where $\gamma_1=c$.
Construct the new unicyclic graph $U'$ in $\mathcal{U}_{n,\Delta}$ by removing the rooted tree $T_c$ from $U$ and replacing it with the path $P_c$. 
It follows that, $GRM_{\lambda}(U)\ge GRM_{\lambda}(U')$ and $\deg_{U'}(c)=3$. Construct the unicyclic graph $U''\in \mathcal{U}_{n,\Delta}$ by removing the vertices $\gamma_2,\cdots, \gamma_k$ from $U'$ and replacing the edge $cc_1$ by the path $c \gamma_2 \cdots \gamma_k c_1 $. 
Consequently, we have: $$\deg_{U''}(c_1)=\deg_{U'}(c_1)=\deg_{U}(c_1),$$ $$\deg_{U''}(c_2)=\deg_{U'}(c_2)=\deg_{U}(c_2).$$
In the subsequent computations we set $X = GRM_{\lambda}(U')-GRM_{\lambda}(U'')$. 

If the length of $P_{c}$ is greater than one, then 
\begin{align*}
X = &\ (\deg_{U'}(c_1)+\lambda)(\deg_{U'}(c)+\lambda)+(\deg_{U'}(c_2)+\lambda)(\deg_{U'}(c)+\lambda)
\\
&+(\deg_{U'}(\gamma_2)+\lambda)(\deg_{U'}(c)+\lambda)+(\deg_{U'}(\gamma_{k-1})+\lambda)(\deg_{U'}(\gamma_{k})+\lambda) \\
&-(\deg_{U''}(c_1)+\lambda)(\deg_{U''}(\gamma_{k})+\lambda)-(\deg_{U''}(c_2)+\lambda)(\deg_{U''}(c)+\lambda)\\
&-(\deg_{U''}(\gamma_{2})+\lambda)(\deg_{U''}(c)+\lambda)-(\deg_{U''}(\gamma_{k-1})+\lambda)(\deg_{U''}(\gamma_{k})+\lambda)\\
=&\ (\deg_{U}(c_1)+\lambda)(3+\lambda)+(\deg_{U}(c_2)+\lambda)(3+\lambda)
\\
&+(2+\lambda)(3+\lambda)+(2+\lambda)(1+\lambda) \\
&-(\deg_{U}(c_1)+\lambda)(2+\lambda)-(\deg_{U}(c_2)+\lambda)(2+\lambda)
-2(2+\lambda)^2\\
=&\ \deg_{U}(c_1)+\deg_{U}(c_2)+2\lambda>0.
\end{align*}
If, however, the length of $P_{c}$ is one, then
\begin{align*}
X  = &\ (\deg_{U'}(c_1)+\lambda)(\deg_{U'}(c)+\lambda)+(\deg_{U'}(c_2)+\lambda)(\deg_{U'}(c)+\lambda)
\\
&+(\deg_{U'}(\gamma_{2})+\lambda)(\deg_{U'}(c)+\lambda)\\
&-(\deg_{U''}(c_1)+\lambda)(\deg_{U''}(\gamma_{2})+\lambda)-(\deg_{U''}(c_2)+\lambda)(\deg_{U''}(c)+\lambda)\\
&-(\deg_{U''}(\gamma_{2})+\lambda)(\deg_{U''}(c)+\lambda)\\
=&\ (\deg_{U}(c_1)+\lambda)(3+\lambda)+(\deg_{U}(c_2)+\lambda)(3+\lambda)+(1+\lambda)(3+\lambda)
\\
&-(\deg_{U}(c_1)+\lambda)(2+\lambda)-(\deg_{U}(c_2)+\lambda)(2+\lambda)-(2+\lambda)^2\\
=&\ \deg_{U}(c_1)+\deg_{U}(c_2)+2\lambda-1>0.
\end{align*}

Finally, suppose $\deg_{U}(b)\ge4$, with $b_1$, $b_2$, and $b_3$ as neighbors of $b$ in $U$, where $b_1$ and $b_2$ are on the cycle $C$, and $b_3$ is on the path $P$. Define $T_b$ as the rooted tree with the maximum number of vertices connected to $b$, excluding $b_1$, $b_2$, and $b_3$. By Lemma \ref{lem:key}, we can transform $T_b$ into a path $P_{b}$ with the same number of vertices, such that $GRM_{\lambda}(T_{b})\ge GRM_{\lambda}(P_{b})$. Let $P_b=\beta_1\beta_2 \cdots \beta_t$, where $\beta_1=b$.
Construct the new unicyclic graph $U'$ in $\mathcal{U}_{n,\Delta}$ by removing the rooted tree $T_{b}$ from $U$ and replacing it with the path $P_{b}$. 
It follows that $GRM_{\lambda}(U)\ge GRM_{\lambda}(U')$ and $\deg_{U'}(b)=4$. Construct the unicyclic graph $U''$ in $\mathcal{U}_{n,\Delta}$ by removing the vertices $\beta_2,\cdots,\beta_t$ from $U'$ and replacing the edge $bb_1$ with the path $b \beta_2\cdots \beta_t b_1$. 
Consequently, we have: $$\deg_{U''}(b_1)=\deg_{U'}(b_1)=\deg_{U}(b_1),$$ $$\deg_{U''}(b_2)=\deg_{U'}(b_2)=\deg_{U}(b_2),$$ $$\deg_{U''}(b_3)=\deg_{U'}(b_3)=\deg_{U}(b_3).$$ If the length of $P_{b}$ is at least two, then
\begin{align*}
X = &\ (\deg_{U'}(b_1)+\lambda)(\deg_{U'}(b)+\lambda)+(\deg_{U'}(b_2)+\lambda)(\deg_{U'}(b)+\lambda)
\\
&+(\deg_{U'}(b_3)+\lambda)(\deg_{U'}(b)+\lambda)+(\deg_{U'}(\beta_2)+\lambda)(\deg_{U'}(b)+\lambda)\\&
+(\deg_{U'}(\beta_{t-1})+\lambda)(\deg_{U'}(\beta_t)+\lambda) \\
&-(\deg_{U''}(b_1)+\lambda)(\deg_{U''}(\beta_t)+\lambda)-(\deg_{U''}(b_2)+\lambda)(\deg_{U''}(b)+\lambda)\\
&-(\deg_{U''}(b_3)+\lambda)(\deg_{U''}(b)+\lambda)-(\deg_{U''}(\beta_2)+\lambda)(\deg_{U''}(b)+\lambda)\\&
-(\deg_{U''}(\beta_{t-1})+\lambda)(\deg_{U''}(\beta_t)+\lambda)\\
=&\ (\deg_{U}(b_1)+\lambda)(4+\lambda)+(\deg_{U}(b_2)+\lambda)(4+\lambda)
\\
&+(\deg_{U}(b_3)+\lambda)(4+\lambda)+(2+\lambda)(4+\lambda)+(2+\lambda)(1+\lambda) \\
&-(\deg_{U}(b_1)+\lambda)(2+\lambda)-(\deg_{U}(b_2)+\lambda)(3+\lambda)
\\
&-(\deg_{U}(b_3)+\lambda)(3+\lambda)-(2+\lambda)(3+\lambda)-(2+\lambda)^2\\
=&\ 2\deg_{U}(b_1)+\deg_{U}(b_2)+\deg_{U}(b_3)+4\lambda>0.
\end{align*}
If, however, the length of $P_{b}$ is one, then
\begin{align*}
X = &\ (\deg_{U'}(b_1)+\lambda)(\deg_{U'}(b)+\lambda)+(\deg_{U'}(b_2)+\lambda)(\deg_{U'}(b)+\lambda)
\\
&+(\deg_{U'}(b_3)+\lambda)(\deg_{U'}(b)+\lambda)+(\deg_{U'}(\beta_2)+\lambda)(\deg_{U'}(b)+\lambda) \\
&-(\deg_{U''}(b_1)+\lambda)(\deg_{U''}(\beta_2)+\lambda)-(\deg_{U''}(b_2)+\lambda)(\deg_{U''}(b)+\lambda)\\
&-(\deg_{U''}(b_3)+\lambda)(\deg_{U''}(b)+\lambda)-(\deg_{U''}(\beta_2)+\lambda)(\deg_{U''}(b)+\lambda)\\
=&\ (\deg_{U}(b_1)+\lambda)(4+\lambda)+(\deg_{U}(b_2)+\lambda)(4+\lambda)
\\
&+(\deg_{U}(b_3)+\lambda)(4+\lambda)+(1+\lambda) (4+\lambda)\\
&-(\deg_{U}(b_1)+\lambda)(2+\lambda)-(\deg_{U}(b_2)+\lambda)(3+\lambda)
\\
&-(\deg_{U}(b_3)+\lambda)(3+\lambda)-(2+\lambda)(3+\lambda)\\
=&\ 2\deg_{U}(b_1)+\deg_{U}(b_2)+\deg_{U}(b_3)+4\lambda-2>0.
\end{align*}
This concludes the proof of Lemma~\ref{L4}.
\end{proof}

As the proof of the next lemma closely resembles the one presented in Lemma \ref{L4}, it is omitted for brevity.

\begin{lemma}\label{L5}
Assuming $\lambda\ge -\frac{1}{2}$, consider a graph $U$ in $\mathcal{U}_{n,\Delta}$ with a vertex $a$ of degree $\Delta$ on its cycle. If $U$ contains another vertex $b$ (excluding $a$) of a degree of at least 3, then there exists a unicyclic graph $U'$ in $\mathcal{U}_{n,\Delta}$ where $GRM_{\lambda}(U)>GRM_{\lambda}(U')$.
\end{lemma}

\begin{lemma}\label{L6}
Let $\lambda\ge -\frac{1}{2}$ and let $U \in \mathcal{U}_{n,\Delta}$ have the vertex $a$ of degree $\Delta$ that is not on its cycle. 
If $a$ is adjacent to at least two vertices of degree more than one, then $\mathcal{U}_{n,\Delta}$ contains a unicyclic graph $U'$ such that $GRM_{\lambda}(U)>GRM_{\lambda}(U')$.
\end{lemma}
\begin{proof}
Let $b$ be a vertex on the cycle of $U$ minimizing the distance $d_{U}(a,b)$ and let $P$ be the unique path connecting $a$ and $b$. Assume that $T_a$ is the tree rooted at $a$ with the maximum possible number of vertices that are connected to $a$ by paths, and it satisfies $deg_{T_a}(a)=\Delta-1$ and $V(T_a)\cap V(P)=\{a\}$. According to the transformations described in Lemma \ref{lem:key}, we can transform the tree $T_a$ into a spider tree $T_{a}'$ with the same number of vertices, the same root $a$, and $\Delta-1$ legs, such that $V(T_a')\cap V(P)=\{a\}$ and $GRM_{\lambda}(T_{a})\ge GRM_{\lambda}(T_{a}')$. Now, consider the unicyclic graph $U'$ in $\mathcal{U}_{n,\Delta}$ obtained by replacing $T_{a}$ with $T_{a}'$. 
Consequently, it holds that $GRM_{\lambda}(U)\ge GRM_{\lambda}(U')$. 

Suppose $a$ is adjacent to at least two vertices of degree greater than one. By Theorem \ref{thm1}, $T_{a}'$ takes the form of a spider graph with one leg longer than one, such as $P'= c_1\,c_2\,\dots c_k$. Let $c_1$ be a vertex in $N_{U'}(a)\cap V(P')$, and select $d\in N_{U'}(a) \setminus \{c_1\}$ so that $d$ is on the path $P$. Now, derive a new unicyclic graph $U''$ in $\mathcal{U}_{n,\Delta}$ by removing $c_2,\,\dots,c_k$ from $U'$ and adding the path $d\,c_2\,\dots c_k\,a$. According to Lemma \ref{L4}, we can assume that
$\deg_{U''}(d)=\deg_{U'}(d)=2$ when $d\neq b$ and $\deg_{U''}(d)=\deg_{U'}(d)=3$ when $d=b$. 
If $\deg_{U''}(d)=\deg_{U'}(d)=2$, then recalling that $X = GRM_{\lambda}(U')-GRM_{\lambda}(U'')$, we have 
\begin{align*}
X = &\ 2(2+\lambda)(\Delta+\lambda)+(2+\lambda)(1+\lambda)
\\
&-(2+\lambda)(\Delta+\lambda)-(1+\lambda)(\Delta+\lambda)-(2+\lambda)(2+\lambda)\\
=&\ \Delta-2>0,
\end{align*}
and if $\deg_{U''}(d)=\deg_{U'}(d)=3$, then 
\begin{align*}
X = &\ (3+\lambda)(\Delta+\lambda)+(2+\lambda)(\Delta+\lambda)+(2+\lambda)(1+\lambda)
\\
&-(2+\lambda)(\Delta+\lambda)-(1+\lambda)(\Delta+\lambda)-(3+\lambda)(2+\lambda)\\
=&\ 2\Delta-4>0,
\end{align*}
which proves Lemma~\ref{L6}.
\end{proof}

The proof of the next lemma closely resembles the one presented in Lemma \ref{L6}, so it will be omitted. 

\begin{lemma}\label{L7}
Let $\lambda\ge -\frac{1}{2}$ and let $U \in \mathcal{U}_{n,\Delta}$ have the vertex $a$ of degree $\Delta$ on its cycle. If $a$ is adjacent to at least three vertices of degree greater than one, then $\mathcal{U}_{n,\Delta}$ contains a unicyclic graph $U'$ such that $GRM_{\lambda}(U)>GRM_{\lambda}(U')$.
\end{lemma}

From Lemmas \ref{L4}–\ref{L7}, we verify that within the set $\mathcal{U}_{n,\Delta}$, the minimal unicyclic graphs with respect to $GRM_{\lambda}$ must either belong to the subset $\mathcal{U}^{(2)}_{n,\Delta}$, or coincide with one of the graphs ${U}^{(1)}_{n,\Delta}$ or ${U}^{(3)}_{n,\Delta}$. These structures are formally defined as follows:

\begin{itemize}
\item
Let ${U}^{(1)}_{n,\Delta}$ be a unicyclic graph from $\mathcal{U}_{n,\Delta}$ with a vertex $a$ of degree $\Delta$, where $a$ is adjacent to $\Delta-1$ leaves and $a$ is not on the cycle of ${U}^{(1)}_{n,\Delta}$. Additionally, let $b$ be a vertex 
on the cycle of ${U}^{(1)}_{n,\Delta}$ such that $d_{{U}^{(1)}_{n,\Delta}}(a,b)=1$, $\deg_{{U}^{(1)}_{n,\Delta}}(b)=3$, and for every vertex $c\in V({U}^{(1)}_{n,\Delta})\setminus \{a,b\}$, it holds that $1\le \deg_{{U}^{(1)}_{n,\Delta}}(c)\le2$ (see Fig. \ref {fig:U1}). Then
\begin{align}\label{item1}
	GRM_{\lambda}({U}^{(1)}_{n,\Delta})=&2(3+\lambda)(2+\lambda)+(\Delta-1)(1+\lambda)(\Delta+\lambda)+(3+\lambda)(\Delta+\lambda)\notag\\&+(n-\Delta-2)(2+\lambda)^2\notag \\
=& (\Delta+\lambda)(\Delta+\Delta\lambda+2)+(n-\Delta)(2+\lambda)^2+2(2+\lambda). 
\end{align}	

\begin{figure}[h!] \centering
		\includegraphics[width=2in,height=1in]{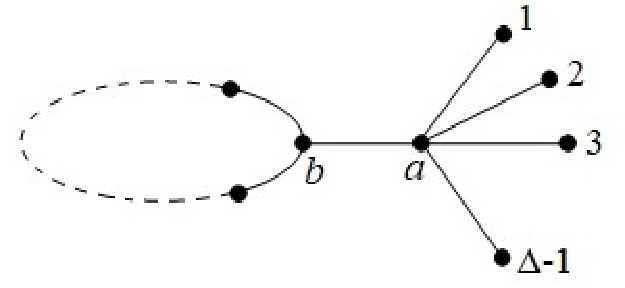}	
		\caption{The unicyclic graph ${U}^{(1)}_{n,\Delta}$.}
\label{fig:U1}	
	\end{figure}
	
\item 
Let $\mathcal{U}^{(2)}_{n,\Delta}$ be a subset of $\mathcal{U}_{n,\Delta}$ such that every unicyclic graph $U_2\in\mathcal{U}^{(2)}_{n,\Delta}$ has a vertex $a$ of degree $\Delta$, where $a$ is adjacent to $\Delta-1$ leaves and is not part of the cycle. Additionally, assume $b$ is a vertex 
on the cycle of $U_2$ such that the distance $d_{U_2}(a,b)$ is minimized, specifically with $d_{U_2}(a,b)\ge2$, $\deg_{U_2}(b)=3$, and for every vertex $c\in V(U_2)\setminus \{a,b\}$, it holds that $1\le \deg_{U_2}(c)\le2$ (see Fig.~\ref{fig:U2}). Therefore, for every unicyclic graph $U_2\in\mathcal{U}^{(2)}_{n,\Delta}$, we have
\begin{align}\label{item2}
GRM_{\lambda}(U_2)=&3(3+\lambda)(2+\lambda)+(\Delta-1)(1+\lambda)(\Delta+\lambda)+(2+\lambda)(\Delta+\lambda)\notag\\
&+(n-\Delta-3)(2+\lambda)^2\notag \\
=& (\Delta+\lambda)(\Delta+\Delta\lambda+1)+(n-\Delta)(2+\lambda)^2+3(2+\lambda).
\end{align}

\item
Let $U^{(3)}_{n,\Delta}$ be a unicyclic graph from $\mathcal{U}_{n,\Delta}$ with a vertex $a$ on its cycle, where $a$ has degree $\Delta$ and is adjacent to $\Delta-2$ leaves. Additionally, for every vertex $b\in V(U^{(3)}_{n,\Delta})\setminus\{a\}$, it holds that $1\le \deg_{U^{(3)}_{n,\Delta}}(b)\le2$ (see Fig.~\ref{fig:U3}). Then

\begin{align}\label{item3}
GRM_{\lambda}(U^{(3)}_{n,\Delta})=&(\Delta-2)(1+\lambda)(\Delta+\lambda)+2(2+\lambda)(\Delta+\lambda)+(n-\Delta)(2+\lambda)^2 \notag\\
=&(\Delta+\lambda)(\Delta+\Delta\lambda+2)+(n-\Delta)(2+\lambda)^2.
\end{align}	
\end{itemize}

Now, we compare $GRM_{\lambda}$ values across these subsets.

\begin{lemma}\label{L8}
For each unicyclic graph $U_2\in\mathcal{U}^{(2)}_{n,\Delta}$, we have $$GRM_{\lambda}({U}^{(1)}_{n,\Delta})>GRM_{\lambda}(U_2).$$
\end{lemma}
\begin{proof} By subtracting Eq.~\eqref{item2} from Eq.~\eqref{item1}, we obtain
$$GRM_{\lambda}({U}^{(1)}_{n,\Delta})-GRM_{\lambda}(U_2)=\Delta-2>0,$$
from which the result follows.
\end{proof}

\begin{lemma}\label{L8-1}
The following inequality holds: 
$$GRM_{\lambda}({U}^{(1)}_{n,\Delta})>GRM_{\lambda}({U}^{(3)}_{n,\Delta}).$$ 
\end{lemma}

\begin{proof} By subtracting Eq.~\eqref{item3} from Eq.~\eqref{item1}, we obtain
$$GRM_{\lambda}({U}^{(1)}_{n,\Delta})-GRM_{\lambda}({U}^{(3)}_{n,\Delta})=2(2+\lambda)>0,$$
from which the result follows.
\end{proof}

\begin{lemma}\label{L9}
For each unicyclic graph $U_2\in \mathcal{U}^{(2)}_{n,\Delta}$, the following inequalities hold: 

If $\Delta>2\lambda+6$, then $GRM_{\lambda}({U}^{(3)}_{n,\Delta})>GRM_{\lambda}(U_2)$. 

If $\Delta < 2\lambda+6$, 
then $GRM_{\lambda}(U_2)>GRM_{\lambda}({U}^{(3)}_{n,\Delta})$.

If $\Delta=2\lambda+6$, then $GRM_{\lambda}({U}^{(3)}_{n,\Delta})=GRM_{\lambda}(U_2)$.
\end{lemma}
\begin{proof} By subtracting Eq.~\eqref{item2} from Eq.~\eqref{item3}, we obtain:
$$GRM_{\lambda}({U}^{(3)}_{n,\Delta})-GRM_{\lambda}(U_2)=\Delta-2\lambda-6.$$
It can now be easily verified that, if $\Delta>2\lambda+6$, then $GRM_{\lambda}({U}^{(3)}_{n,\Delta})>GRM_{\lambda}(U_2)$, if $\Delta<2\lambda+6$, then $GRM_{\lambda}(U_2)>GRM_{\lambda}({U}^{(3)}_{n,\Delta})$, and if $\Delta=2\lambda+6$, then $GRM_{\lambda}({U}^{(3)}_{n,\Delta})=GRM_{\lambda}(U_2)$.
\end{proof}

With all the prerequisites in place, we can now proceed to establish the proof of Theorem~\ref{thm3}. 

Assume that $U'\in \mathcal{U}_{n,\Delta}$ with $GRM_{\lambda}(U)\ge GRM_{\lambda}(U')$ 
for every $U \in \mathcal{U}_{n,\Delta}$. By Lemmas \ref{L4}--\ref{L9}, $U'$ satisfies one of the following cases:

\noindent{\bf Case 1.} 
If $ 2\lambda+6<\Delta<n-3$, then $U'\in \mathcal{U}^{(2)}_{n,\Delta}$, and from Eq.~\eqref{item2}, we obtain: $$GRM_{\lambda}(U')= (\Delta+\lambda)(\Delta+\Delta\lambda+1)+(n-\Delta)(2+\lambda)^2+3(2+\lambda).$$ So the inequality in Theorem~\ref{thm3}~(i) holds
with equality if and only if $U\in \mathcal{U}^{(2)}_{n,\Delta}$.  

\noindent{\bf Case 2.} 
If $3 \leq \Delta < 2\lambda+6$ or $n-3 \leq \Delta \leq n-1$, then $U'= {U}^{(3)}_{n,\Delta}$, and from Eq.~\eqref{item3}, we have:
$$GRM_{\lambda}(U')=(\Delta+\lambda)(\Delta+\Delta\lambda+2)+(n-\Delta)(2+\lambda)^2.$$
Hence the inequality in Theorem~\ref{thm3}~(ii) holds with equality if and only if $U={U}^{(3)}_{n,\Delta}$.  

\noindent{\bf Case 3.} 
If $\Delta =2\lambda+6$, then $U'\in \mathcal{U}^{(2)}_{n,\Delta}$ or $U'={U}^{(3)}_{n,\Delta}$, and from Eq.~\eqref{item2} or Eq.~\eqref{item3}, we get:
$$GRM_{\lambda}(U')= 6(2+\lambda)^3+(n-2\lambda-6)(2+\lambda)^2.$$
Then the inequality in Theorem~\ref{thm3}~(iii) holds with equality if and only if either $U\in \mathcal{U}^{(2)}_{n,\Delta}$ or $U={U}^{(3)}_{n,\Delta}$.  
This completes the proof of Theorem~\ref{thm3}.

\medskip
We end this section with providing a comparison between the bound of Theorem~\ref{th2} and that of Theorem~\ref{thm3}.

If $ 2\lambda+6<\Delta<n-3$, then 
\begin{align*}
&\ (\Delta+\lambda)(\Delta+\Delta\lambda+1)+(n-\Delta)(2+\lambda)^2+3(2+\lambda)-
n(2+\lambda)^2\\
=&\ (1+\lambda)(\Delta^2-3\Delta+4)+2>0.
\end{align*}

If $3 \leq \Delta \leq 2\lambda+6$ or $n-3 \leq \Delta \leq n-1$, then 
\begin{align*}
&\ (\Delta+\lambda)(\Delta+\Delta\lambda+2)+(n-\Delta)(2+\lambda)^2-n(2+\lambda)^2\\
= &\ (1+\lambda)(\Delta^2-3)+2\lambda+\Delta>0.
\end{align*} 
The above comparison shows that the lower bound of Theorem~\ref{thm3} is stronger than that of Theorem~\ref{th2}, where $\lambda\ge -\frac{1}{2}$.

\section{Concluding remarks}

In this paper, we have extended and refined the bounds established in Theorem~\ref{th1} and Theorem~\ref{th2} by deriving sharp lower bounds for the general reduced second Zagreb index of trees and unicyclic graphs with a specified order and maximum degree. Furthermore, we recall the following result.

\begin{prelem}\label{HHH}{\rm\cite{HHH}}
  If $\Gamma$ is a connected graph and $\lambda\ge -\frac{1}{2}$, then for every $e\notin E(\Gamma)$,
$$GRM_\lambda(\Gamma+e)>GRM_\lambda(\Gamma).$$
\end{prelem}

This result thus asserts that adding an edge to a connected graph increases the value of the general reduced second Zagreb index. By combining Theorem~\ref{HHH} with Theorem~\ref{thm1}, we can formulate the following additional result.

\begin{theorem}\label{thm2}
If $\lambda\ge -\frac{1}{2}$ and $\Gamma$ is a connected graph of order $n\ge 3$ and maximum degree $\Delta$, then 
$$GRM_{\lambda}(\Gamma)\geqslant \left\{\begin{array}{lc}
n(2+\lambda)^2-3(2+\lambda)+(1+\lambda)(\Delta^2-3\Delta-\lambda);& \Delta< n-1,\\
&\\
 \Delta(\Delta+\lambda)(1+\lambda);&
\Delta=n-1.
\end{array}\right.$$
The equality condition is met if and only if the graph $\Gamma$ is a spider graph with at most one leg of length exceeding one.
\end{theorem}

We have also identified the minimal trees and unicyclic graphs that achieve these lower bounds in our main two theorems. The identification of minimal trees and unicyclic graphs provides a framework for future research and applications in mathematical chemistry and network theory.
Future research may explore further generalizations of our findings, as well as their implications for other classes of graphs. Additionally, investigating the behavior of the general reduced second Zagreb index under various graph operations could yield new avenues for understanding the structural properties of graphs.

\section*{Acknowledgments}
The authors would like to thank the Editor and the anonymous reviewers for their very careful reading of the manuscript and for their valuable comments and suggestions, which significantly improved the clarity and presentation of the paper.\\
Sandi Klav\v zar were supported by the Slovenian Research and Innovation Agency (ARIS) under the grants P1-0297, N1-0355, and N1-0285.

\section*{Author Contributions} All authors contributed equally to the conception, design, investigation, methodology, validation, resources, and writing of this manuscript. All authors have read and approved the final version of the manuscript.

\section*{Availability of data and materials} The paper includes or uses no datasets.

\section*{Conflicts of Interest}
The authors declare that they have no conflict of interest.

\end{document}